\documentclass[authoryear,preprint,review,12pt]{elsarticle}
\let\today\relax
\makeatletter
\def\ps@pprintTitle{%
    \let\@oddhead\@empty
    \let\@evenhead\@empty
    \def\@oddfoot{\footnotesize\itshape
         {} \hfill\today}%
    \let\@evenfoot\@oddfoot
    }
\makeatother

\usepackage{times}
\usepackage{basic}

\usepackage{amssymb,amsmath,amsfonts,amsthm}
\newtheorem{theorem}{Theorem}[section]
\newtheorem{definition}[theorem]{Definition}
\newtheorem{lemma}[theorem]{Lemma}
\newtheorem{assumption}[theorem]{Assumption}

\newtheorem{remark}[theorem]{Remark}

\begin{document}
\begin{frontmatter}
\title{Classification of sparse binary vectors}
\author[add1]{Evgenii Chzhen}
\ead{evgenii.chzhen@univ-paris-est.fr}

\address[add1]{LAMA, Universit\'e Paris-Est}

\begin{abstract}
In this work we consider a problem of multi-label classification, where each instance is associated with some binary vector.
Our focus is to find a classifier which minimizes false negative discoveries under constraints.
Depending on the considered set of constraints we propose plug-in methods and provide non-asymptotic analysis under margin type assumptions.
Specifically, we analyze two particular examples of constraints that promote sparse predictions: in the first one, we focus on classifiers with $\ell_0$-type constraints and in the second one, we address classifiers with bounded false positive discoveries.
Both formulations lead to different Bayes rules and, thus, different plug-in approaches.
The first considered scenario is the popular multi-label top-$K$ procedure: a label is predicted to be relevant if its score is among the $K$ largest ones.
For this case, we provide an excess risk bound that achieves so called ``fast'' rates of convergence under a generalization of the margin assumption to this settings.
The second scenario differs significantly from the top-$K$ settings, as the constraints are distribution dependent.
We demonstrate that in this scenario the almost sure control of false positive discoveries is impossible without extra assumptions.
To alleviate this issue we propose a sufficient condition for the consistent estimation and provide non-asymptotic upper bound.
\end{abstract}


\end{frontmatter}
\section{Introduction}
\label{sec:introduction}
The goal of the multi-label classification is to annotate an observed object with a set of relevant labels.
Such a task encompasses a number of applications, for instance, text categorization~\citep{Gao_Wu_Lee_Chua04}, functional genomics~\citep{Barutcuoglu_Schapire_Troyanskaya06}, and image classification~\citep{li_Xim_Zhao_Feipeng_GUo_Yuhong14}.
Several sophisticated algorithms have been recently developed, including tree based algorithms~\citep{Jain_Prabhu_Varma16} and embedding based algorithms~\citep{Yu_Jain_Kar_Dhillon14, Bhatia_Jain_Kar_Varma_Jain15} which are considered to be state-of-the-art.
Other contributions have rather focused on efficient implementations of existing multi-label strategies: for instance in~\citep{Babbar_Scholkopf17} the authors developed a large-scale distributed framework relying on one-versus-rest strategy applied to linear classifiers, plug-in type classifiers were considered in~\citep{dembczynski2013optimizing}.

A consensus on the choice of the performance measure is still missing.
Yet, most recent works have pointed out that it is more rewarding to correctly predict a relevant\footnote{A label is called relevant for an instance if this instance is tagged with this label.} label than to give a correct prediction on irrelevant labels, see~\citep{Jain_Prabhu_Varma16} for a thorough discussion on this topic.
Such asymmetry is usually explained by the label space sparsity, that is, there is only a small set of relevant labels compared to the set of irrelevant ones.
It also suggests that the classical Hamming loss is not well tailored for sparse multi-label problems as it treats both false positives and false negatives equally; thus, some modifications ought to be proposed.

To introduce this asymmetric information in a learning algorithm, one can modify the objective loss function to be minimized.
For instance, in~\citep{Jain_Prabhu_Varma16} the authors have weighted each label, according to their observed frequency over a dataset.
These weights are motivated by the propensity model, which introduces a possibility of non-observing a relevant label.
To be more precise, \citet{Jain_Prabhu_Varma16} propose to down-weight the reward for correctly predicting a frequent label, which is motivated by the observation that the frequent labels can be easily predicted by a human.
In~\citep{Chzhen_Denis_Hebiri_Salmon17}, the authors proposed to weight false positive (irrelevant labels predicted to be relevant) and false negative (relevant labels predicted to be irrelevant) discoveries separately. The empirical risk minimization procedure was then analyzed thanks to Rademacher's complexity techniques.

Another possible direction is to consider a more complex family of loss functions, which are called non-decomposable, such as $\text{F}_1$-score or AUC among others.
A general class of loss functions which can be represented as a ratio of false discoveries is studied in~\citep{Koyejo_Natarajan_Ravikumar_Dhillon15}.
\citet{Koyejo_Natarajan_Ravikumar_Dhillon15} showed that the oracle (Bayes optimal) classifier can be obtained by thresholding the regression functions associated with each label, that is, the probability of a label to be relevant.
Additionally, the authors proved that algorithms based on plug-in are consistent and have a good empirical performance.
In a similar direction, \citet{dembczynski2013optimizing} empirically showed that plug-in algorithms outperform the ones based on the structured loss minimization, in the context of multi-label classification with $\text{F}_1$-score performance measure.
\citet{dembczynski2013optimizing} additionally established a statistical consistency of the considered algorithms.
Finally, convex empirical risk minimization was studied in~\citep{Gao_Zhou11}, where authors proved an infinite sample size consistency for convexified Hamming loss and ranking loss.
Consistency results are common in the multi-label classification literature.
Though, results of non-asymptotic nature, \eg excess risk bounds, have not received much attention in these settings.

Due to the sparse nature of the problem we propose to focus on classifiers that minimize false negative discoveries and exhibit desirable structural properties.
This can be seen as a problem constrained estimation, mainly considered in the settings of regression or parametric estimation~\citep{Lepski_89}.
In the constrained estimation, similarly to this case, the goal is to find an estimation which inherit some properties desired by a statistician.
In this work we consider two particular choices of structural constraints.
The first type of constraints describes classifiers with a bounded number of predicted labels: for instance, this approach appears naturally in recommendation systems.
Bayes optimal classifier in this context is given by the top-$K$ procedure, popular among practitioners: a label is predicted to be relevant if its associated score is among the top-$K$ values.
The popularity of this approach is reflected by several recent works~\citep{lapin15,Li2017ImprovingPR}, where top-$K$ procedures are studied both from  applied and asymptotic points of view.
In contrast, for this scenario, we establish a non-asymptotic excess risk bound for plug-in based classifiers.
The obtained bound can attain ``fast'' and ``super-fast'' (faster than $1/N$) rates under a multi-label version margin assumption.
Moreover, the bound is shown to be optimal in the minimax sense when instantiated to binary classification~\citep{Audibert_Tsybakov07}.

For the second scenario, we consider a set of classifiers with a control over false positive discoveries.
This can be relevant when one can tolerate a few false positive discoveries, but needs a parameter which quantifies the level of tolerance.
To provide guarantees for this instance, we introduce a different set of assumptions which reflects the label sparsity of a typical multi-label problem.
Under these assumptions, we prove an excess risk bound similar (in terms of  rates) to the bound obtained by~\citet{Denis_Hebiri15}, where the authors analyzed a binary classification framework with a control over the probability of rejection.

This paper is organized in the following way:
in Section~\ref{sec:notation}, we introduce notation used throughout the paper, formally state the considered framework and lay down important results that we use.
Further, Sections~\ref{sec:k-sparse} and~\ref{sec:contr_mist_zero} are devoted to the theoretical analysis of plug-in rules in the two scenarios mentioned above.
We conclude the paper by a discussion on possible extensions in Section~\ref{sec:discussion_extensions}.
All the proofs are gathered in Appendix.

\section{Framework and notation}
\label{sec:notation}
In this section, we introduce the notation used in our work and present our proposed framework.
For any positive integer number $N$ we denote by $[N] = \{1, \ldots, N\}$ the set of integers between $1$ and $N$.
For any vector $a$ in a Euclidean space $\bbR^N$ and for all $i \in [N]$ we denote by $a^i$ the $i^{\text{th}}$ component of the vector $a$.
We denote by $\normin{\cdot}_0$ the $\ell_0$ norm of a vector, which in case of binary vectors reduces to the number of ones.
For every real numbers $a, b$ we denote by $a \wedge b$ the minimum between $a$ and $b$.
Let $(X, Y) \sim \Prob$, where $X \in \spacefeature  = \bbR^D$ and $Y = (Y^1, \ldots, Y^L)^\top \in \spacelabel = \binvec$.
Denote by $\Prob_{X}$ the marginal distribution of $X$.
A classifier $f = (f^1, \ldots, f^L)^\top$ is a measurable function from $\spacefeature$ to $\spacelabel$, that is $f: \spacefeature \mapsto \spacelabel$, and we write $\classall$ for the set of all classifiers (measurable functions).
Let $\eta(x) = (\eta^1(x), \ldots, \eta^L(x))^\top : \spacefeature \mapsto [0, 1]^L$ be the component wise regression function, meaning that for all $l \in [L]$ the $l^{\text{th}}$ component of $\eta(x)$ is given by $\eta^l(x) = \Prob(Y^l = 1 | X=x)$.
We denote by $\sigma = (\sigma_1,\dots,\sigma_L)$ a permutation\footnote{we omit the dependence on $x$ and write $\sigma$ instead of $\sigma(x)$.} of $[L]$ such that the regression functions is ranked as
    \[
        \eta^{\sigma_1}(x) \geq \ldots\geq \eta^{\sigma_L}(x)\enspace,
    \]
for all $x \in \bbR^D$.
The average false negative risk of a classifier $f \in \classall$ is denoted by
\begin{align}
    \label{eq:risk}
    \risk(f) = \frac{1}{L}\sum\limits_{l = 1}^L\Prob\{f^l(X) = 0, Y^l = 1\}\enspace.
\end{align}
For a fixed subset of predictors $\classgeneric \subset \classall$, specified according to the context, we define an $\classgeneric$-oracle classifier as
\begin{align}
    \label{eq:oracle}
    \oracle \in \argmin_{f \in \classgeneric} \risk(f)\enspace.
\end{align}
Notice that the $\classgeneric$-oracle rule $\oracle$ depends both on the distribution of $(X, Y)$ and on the set of predictors $\classgeneric$.
However, we suppress the explicit dependence on $\classgeneric$, when no ambiguity occurs.
We assume that the minimum is achieved by a classifier $\oracle \in \classgeneric$, though we do not assume that this classifier is unique.
Intuitively, our framework aims at minimization of the total number of mistakes on $Y^l = 1$ (relevant labels), over a class of prediction rules $\classgeneric$.
For example, the case $\classgeneric = \classall$ leads to an $\classall$-oracle $\oracle \equiv (1, \ldots, 1)^\top$, which reflects a complete tolerance over false positive discoveries.
This simple example shows, that the choice of $\classgeneric$ is a crucial part of the proposed framework.


Given a data sample $\data = \{(X_i, Y_i)\}_{i = 1}^N$, which consists of \iid copies of $(X, Y)$, the goal here is to construct an estimator $\estimator$, based on $\data$, of the $\classgeneric$-oracle $\oracle$.
Estimator $\estimator$ is a function that assigns a classifier to every learning sample $\data$, that is, $\estimator: \cup_{N = 1}^{\infty} (\spacefeature \times \spacelabel)^N \mapsto \classall$.
We denote by $\Probdata$ the product probability measure according to which the
data sample $\data$ is distributed, and by $\Expdata$ the expectation with respect to $\Probdata$.
The goal is to provide non-asymptotic bounds on the excess risk $\Expdata \big[\risk(\estimator)\big] - \risk(\oracle)$.
Additionally, we want our estimate $\estimator$ to satisfy one of the following conditions:
\begin{align}
  \label{eq:requirements}
  \estimator \in \classgeneric,\,\,\forall N\in\bbN;\quad \text{or } \quad \estimator \xrightarrow[N \to \infty]{} f \in \classgeneric\enspace,
\end{align}
where the kind of convergence is to be specified later.
Since, the $\classgeneric$-oracle $\oracle$ is typically available in a closed form and depends on an unknown, in practice, regression vector $\eta(x)$, we consider the plug-in type methods.

A vast amount of literature is focused on the estimation of the regression function $\eta(x)$, that is why this part is not a central object of our study.
In other words, we are rather interested in describing the performance of a classifier based on an arbitrary estimator $\heta(x)$ of the regression function $\eta(x)$ which satisfies for all $l \in [L]$ the following assumption:

\begin{assumption}[Exponential bound]
  \label{ass:exponential_heta}
   For some positive constants $C_1, C_2 > 0$ and $\gamma > 0$, for all $\delta > 0$ and for all $l \in [L]$ we have:
  \begin{align}
      \Probdata\{|\eta^l(x) - \heta^l(x)| \geq \delta\} \leq C_1\exp(-C_2N^{\gamma} \delta^2)\,\,\, \text{a.e. $x \in \bbR^D$ w.r.t. }\Prob_X\enspace.
  \end{align}
\end{assumption}
Such a bound holds for various type of estimators and distributions in both parametric~\citep{Li_Prasad_Ravikumar15} and non-parametric settings~\citep{Audibert_Tsybakov07}.
In non-parametric settings, typically, the parameter $\gamma$ depends on the smoothness of $\eta$ and on the dimension $D$.
Let us notice, that empirical evidences~\citep{dembczynski2013optimizing} suggest to use multinomial logistic regression as an effective estimator for the regression function, though, this estimator might not have the exponential concentration.
The rest of the paper is devoted to theoretical analysis of two specific families  $\classgeneric$.
In both cases we derive the $\classgeneric$-oracle classifier $\oracle$, defined in Eq.~\eqref{eq:oracle}.
Typically, the $\classgeneric$-oracle $\oracle$ depends on the regression function $\eta$, due to the form of the risk considered in Eq.~\eqref{eq:risk}.
Explicit expression for the $\classgeneric$-oracle, provides with a natural motivation to consider plug-in type rules for the construction of $\estimator$.
We establish one of the properties in Eq.~\eqref{eq:requirements} and introduce the set of additional assumptions in order to upper-bound the excess risk.

Let us finish this section with one generic notation used in this work.
For the estimator $\heta(x)$ we denote by $\hsigma = \hsigma(x)$ a permutation of $[L]$ such that the following holds for all $x \in \bbR^D$
    \[
        \heta^{\hsigma_1}(x) \geq \ldots\geq \heta^{\hsigma_L}(x)\enspace,
    \]
we again omit the dependence on $x$ and write $\hsigma$ instead of $\hsigma(x)$.
We reserve $\sigma$ and $\tau$ for a non-decreasing permutations of $\eta(x)$ and $\heta(x)$ respectively.

\section{Control over sparsity}
\label{sec:k-sparse}
In this section, we consider, the set of $K$-sparse classifiers, defined for a fixed $K \in [L]$ as:\
\begin{align}
    \classtop{K} \eqdef \{f \in \classall:\, \forall x\in\bbR^D,\,\, \norm{f(x)}_0 \leq K\}\enspace.
\end{align}
Hence, we are interested in a $K$-sparse classifier, which minimize the total number of mistakes on relevant labels.
It is not hard to see that, a $\classtop{K}$-oracle $\oracle$ is given by the top-$K$ procedure, this is stated formally in the following lemma:
\begin{lemma}[$\classtop{K}$-oracle classifier]
    \label{lem:oracle_sparse}
    An $\classtop{K}$-oracle $\oracle$ can be obtained for all $x \in \bbR^D$ as:
    \begin{align*}
        \oracle^{\sigma_1}(x) &= \ldots= \oracle^{\sigma_K}(x) = 1\enspace,\\
        \oracle^{\sigma_{K+1}}(x) &= \ldots= \oracle^{\sigma_L}(x) = 0\enspace.
    \end{align*}
\end{lemma}
\begin{remark}
    Observe, that in order to recover the $\classtop{K}$-oracle $\oracle$ the only information that is needed is $\{\sigma_1(x),\ldots, \sigma_K(x)\}$.
    In particular, any additional information about the regression vector $\eta(x)$ is not relevant.
\end{remark}
A plug-in strategy $\estimator$ in this case can be defined in a straightforward way for all $x \in \bbR^D$ as:
\begin{align}
    \label{eq:plug-in_rule_sparse}
    \estimator^{\hsigma_1}(x) &= \ldots = \estimator^{\hsigma_{K}}(x) = 1\enspace,\\
    \estimator^{\hsigma_{K+1}}(x) &= \ldots = \estimator^{\hsigma_L}(x) = 0\enspace.
\end{align}
Obviously, the plug-in estimator defined above is exactly a $K$-sparse classifier, that is $\estimator \in \classtop{K}$ for every choice of the data sample $\data$, as required in Eq.~\eqref{eq:requirements}.
Since our goal is to predict as relevant the labels with the top-$K$ probabilities, it is natural to restrict our attention to the distributions for which such top-$K$ labels are well separated.
In this context, we use a top-$K$ margin assumption in the following form:
\begin{assumption}[top-$K$ margin assumption]
    \label{ass:margin_sparse}
    We say that the regression vector $\eta(x)$ satisfies top-$K$ margin assumption, if
    there exist positive constants $C, \alpha$ such that for all $\delta > 0$:
    \[
        \Prob_X\{0 < \eta^{\sigma_K}(X) - \eta^{\sigma_{K+1}}(X) \leq \delta\} \leq C\delta^{\alpha}\enspace.
    \]
\end{assumption}
This assumption is similar to the classical margin assumption used in the context of binary classification.
Under Assumption~\ref{ass:margin_sparse} we can obtain the following bound on the excess risk of $\estimator$:
\begin{theorem}
    \label{main:sparse}
    Under Assumptions~\ref{ass:exponential_heta} and~\ref{ass:margin_sparse}, the excess risk of the plug-in classifier in Eq.~\eqref{eq:plug-in_rule_sparse} can be bounded as follows:
    \[
      \Expdata\big[\risk(\estimator)\big] - \risk(\oracle) \leq \tilde{C}K\frac{L-K}{L}N^{-\tfrac{\gamma(\alpha+1)}{2}}\enspace,
    \]
    for some universal constant $\tilde{C}$.
\end{theorem}
The proof of Theorem~\ref{main:sparse} is based on the following upper bound on the excess risk:
 \begin{align}
        \label{eq:upper_sparse}
         \risk(\estimator) - \risk(\oracle) \leq \Exp_{\Prob_X}\frac{1}{L}\sum\limits_{l=1}^K\sum\limits_{j = K+1}^{L}(\eta^{\sigma_l}(X) - \eta^{\sigma_j}(X))\ind{\estimator^{\sigma_l}(X) = 0, \estimator^{\sigma_j}(X) = 1}\enspace,
\end{align}
which in the case $L = 2$ and $K = 1$ reduces to the classical excess risk in binary classification.
We notice that there are two interesting consequences of this bound: first, the bound can attain ``fast'' ($1/N$) and ``super-fast'' (faster than $1/N$) rates of convergence in terms of $N$, depending on $\gamma$ and $\alpha$;
second, the value of the parameter $K$ (chosen by the practitioner) is often small in applications compared to the total amount of labels $L$.
Hence, the obtained bound illustrates the good performance of the proposed method as it behaves proportionally to $K$ rather than to $L$.
This is crucial when one tries to address scenarios where the total amount of observations $N$ is of the same order as $L$.
Moreover, we expect that the dependence on $K$ and $L$ can be improved or even avoided, since the upper bound on the excess risk in Eq.~\eqref{eq:upper_sparse} is rather rough.

\section{Control over false positives}
\label{sec:contr_mist_zero}
In this section, we consider the set of classifiers with controlled false positive discoveries, defined for a fixed $\beta \in [L]$ as:
\begin{align}
        \classbeta{\beta} \eqdef \Big\{f \in \classall :\, \sum\limits_{l = 1}^L\Prob\{f^l(X) = 1, Y^l = 0| X\} \leq \beta,\,\asx\Big\}\enspace.
\end{align}
\begin{remark}
    \label{rem:on_fp_sp_inclusion}
    One should note that the following inclusion holds for all $\beta \in [L]$:
    \[
        \classtop{\beta} \subset \classbeta{\beta}\enspace,
    \]
    which indicates that the top-$\beta$ strategy controls the false positive discoveries. This is intuitive, as the top-$\beta$ procedure is not making more than $\beta$ false positive discoveries.
    However, this control is not optimal in a situation when a larger (compared to $\beta$) set of labels could be relevant.
    In such a scenario, the $\classbeta{\beta}$-oracle classifier is more advantageous as it is able to output a larger set of potentially relevant labels and still has a controlled false positive discoveries.
\end{remark}
As in the previous section, the $\classbeta{\beta}$-oracle classifier is given by thresholding the top components of the regression function.
However, unlike the previous sparse strategy, the amount of positive components can be different for every $x \in \bbR^D$.
\begin{lemma}[Oracle classifier]
    \label{lem:oracle_mistakes}
    An $\classbeta{\beta}$-oracle $\oracle$ can be obtained for every $x \in \bbR^D$ as
    \begin{align*}
        \oracle^{\sigma_1}(x) &= \ldots = \oracle^{\sigma_{K}}(x) = 1\enspace,\\
        \oracle^{\sigma_{K+1}}(x) &= \ldots = \oracle^{\sigma_L}(x) = 0\enspace,
    \end{align*}
    where $K = K(x)$ is defined as
    \begin{align}
        \label{eq:top_K_X}
        K(x)= \max\big\{m \in [L] : \sum\limits_{l = 1}^m(1 - \eta^{\sigma_l}(x)) \leq \beta\big\}\enspace.
    \end{align}
\end{lemma}
In this case the optimal strategy can be characterized as top-$K(X)$, where $K(X)$ is a random variable defined in Eq.~\eqref{eq:top_K_X}.
Intuitively, for each feature vector $x \in \bbR^D$ the threshold $K(x)$ selects labels with high probability to be relevant, the larger the value $\beta$ (which indicates the higher level of tolerance), the more labels are predicted to be relevant.
\begin{remark}
    Notice, that unlike the previous scenario, it is not sufficient to recover the ordering of the regression function $\eta(x)$ to obtain the $\classbeta{\beta}$-oracle.
    Indeed, due to the definition of $K(X)$, even the knowledge of the whole non-decreasing permutation $\sigma$ is not sufficient without additional information about the components of the regression vector $\eta(x)$.
\end{remark}
Similarly to the previous section, a natural plug-in strategy $\estimator$ reads for all $x \in \bbR^D$:
\begin{align}
    \label{eq:plug-in_rule}
    \estimator^{\hsigma_1}(x) &= \ldots = \estimator^{\hsigma_{\htop}}(x) = 1\enspace,\\
    \estimator^{\hsigma_{\htop+1}}(x) &= \ldots = \estimator^{\hsigma_L}(x) = 0\enspace,
\end{align}
where $\htop = \htop(x)$ is defined as
\begin{align*}
    \htop(x) = \max\big\{m \in [L] : \sum\limits_{l = 1}^m(1 - \heta^{\hsigma_l}(x)) \leq \beta\big\}\enspace.
\end{align*}
Ultimately, to recover the $\classbeta{\beta}$-oracle $\oracle$ we need to estimate both: the non-decreasing permutation $\sigma$ and the regression vector $\eta$.
Since, we do not have an access to neither of those quantities, we use an estimator $\heta$ and it's own non-decreasing permutation $\tau$.
We define $\hclassbeta{\beta}$, replacing $(\eta^l(x))_{l = 1}^L$ by $(\heta^l(x))_{l = 1}^L$ in the definition of $\classbeta\beta$,
in order to prove one of the properties in Eq.~\eqref{eq:requirements}
\begin{definition}[Plug-in $\classbeta{\beta}$-set]
    \label{def:plug-in_beta_accurate_set}
    For every $\beta \in [L]$ we denote the plug-in $\beta$-set as
    \begin{align*}
        \hclassbeta{\beta} \eqdef \Big\{f \in \classall :\, \sum\limits_{l = 1}^L\ind{f^l(X) = 1}(1 - \heta^l(X)) \leq \beta,\,\asx\Big\}\enspace.
    \end{align*}
\end{definition}

Due to the approximation error of $\heta$ the plug-in rule $\estimator$ does not necessary belong to the set $\classbeta{\beta}$ (hence is not comparable to the $\classbeta{\beta}$-oracle $\oracle$).
However, if the estimator $\heta$ of $\eta$ is consistent, then every classifier $f \in \hclassbeta{\beta}$ has asymptotically bounded false positive discoveries on the level $\beta$.
\begin{lemma}[Embedding of the plug-in set]
    \label{lem:embeding_of_plug-in_beta_class}
    There exists $\bbeta$ which satisfies $\bbeta \leq \beta + \sum_{l = 1}^L \norm{\eta^l - \heta^l}_{\infty}$ such that for every $f \in \hclassbeta{\beta}$, we have
    \[
        f \in \classbeta{\bbeta}\enspace.
    \]
    Moreover, under Assumption~\ref{ass:exponential_heta}, with probability at least $1 - \epsilon$ over the dataset $\data$ it holds that
    \[
        \hclassbeta{\beta} \subset \classbeta{\bbeta} \enspace,
    \]
    where $\bbeta = \beta + O(LN^{-\gamma/2}\sqrt{\ln(C_1/\epsilon)})$.
\end{lemma}
Since, clearly, $\estimator \in \hclassbeta{\beta}$ by construction, we establish the second requirement in Eq.~\eqref{eq:requirements}, which is a desired property as we want to restrict our attention to the collection of classifiers $\classbeta{\beta}$.
\begin{assumption}[Local margin assumption]
    \label{ass:separability}
    We say that the regression vector $\eta(x)$ satisfies {local margin} assumption, if there exist constants $C_0 > 0, \alpha_{1} > 0$ such that for all $\delta > 0$, we have
    \begin{align*}
        \Prob_X\big\{\eta^{\sigma_k}(X) - \eta^{\sigma_{k+1}}(X) \leq \delta, K(X) = k\big\} \leq C_0\delta^{\alpha_1}\enspace.
    \end{align*}
\end{assumption}
This assumption states that in the optimal thresholding $K(X) = k$ there is a gap between the $k^{\text{th}}$ and $(k + 1)^{\text{th}}$ regression function, which is similar to Assumption~\ref{ass:margin_sparse}.
This is needed in order to recover the permutation $\sigma$, at least partially until its $K^{\text{th}}$ element.
However, since the amount of labels is not fixed a priori and is itself a random variable, the form of this assumption slightly differs from Assumption~\ref{ass:margin_sparse}.
Additionally one should observe that unlike Assumption~\ref{ass:margin_sparse}, the later restricts the possibility of $\eta^{\sigma_k}(X)$ and $\eta^{\sigma_{k+1}}(X)$ to coincide on a set of large measure, which is similar to~\citep{tsybakov2004}.
\begin{assumption}[Sparsity]
    \label{ass:sparsity_general}
    We say that the regression vector $\eta(x)$ satisfies {sparsity} assumption, if for a positive integer $S$ smaller than $L$, we have
    \[
        \sum\limits_{l = 1}^L \Prob\{Y^l = 1 | X\} \leq S, \quad \asx
    \]
\end{assumption}
This assumption is similar to the one used in~\citep{Chzhen_Denis_Hebiri_Salmon17}, and aims at leveraging sparsity of most real datasets.
It is natural to expect that the value of the sparsity $S$ is smaller than the total amount of labels $L$.
Even though, our analysis does not explicitly assume this relation between $S$ and $L$, bounds that we obtain are more advantageous for such a scenario.
We finally introduce the assumption that is more structural and states that the sum of top regression functions is not too concentrated around $\beta$.
\begin{assumption}[Global margin assumption]
    \label{ass:sparsity_struct}
    We say that the regression vector $\eta(x)$ satisfies {global margin} assumption if, there exists $\alpha_2 > 0$, such that for all $k \geq \beta$, for all $l \in [k]$, and for all $\delta > 0$, we have
    \[
        \Prob_X\Big\{\frac{1}{l}|\sum_{j = k - l + 1}^{k}(1 - \eta^{\sigma_j}(X)) - \beta| \leq \delta, K(X) = k\Big\} \leq \beta^{\alpha_2}\delta^{\alpha_2}h(|k - l|)\enspace,
    \]
    where $h : \bbR_+ \mapsto \bbR_+$ such that $h(0) = 1$ and
     \[
        \sum\limits_{k = \beta}^L \sum\limits_{l = 1}^L h(|k - l|) \leq \tilde{C}(L - \beta)\enspace,
     \]
     for some $\tilde{C} > 0$.
\end{assumption}
The multiplier $\beta^{\alpha_2}$ is due to the fact that the following inequality must always be satisfied for $\delta = \tfrac{1}{k}$ and $l = k$:
\begin{align*}
    \sum\limits_{k = \beta}^L\Prob_X\Big(\frac{1}{k}|\sum_{j =1}^{k}(1 - \eta^{\sigma_j}(X)) - \beta| \leq \frac{1}{k}, K(X) = k\Big)
    &= \sum\limits_{k = \beta}^L\Prob_X\Big(K(X) = k\Big)\\
     &= 1 \leq \sum\limits_{k = \beta}^L \beta^{\alpha_2}\frac{1}{k^{\alpha_2}}h(0)\enspace,
\end{align*}
where the first equality holds since on the event $K(X) = k$ the quantity $|\sum_{j =1}^{k}(1 - \eta^{\sigma_j}(X)) - \beta|$ is always upper bounded by one.
The definition of the function $h$ states that the matrix $H_{k,l} = h(|k - l|)$ is a diagonally dominant matrix such that for all $k \geq \beta$ we have $A H_{k, k} \geq \sum_{l\neq k}H_{k, l}$ for some positive constant $A$ independent from $L$.

Let us provide a simple intuition for the necessity of Assumption~\ref{ass:sparsity_struct}.
Consider the following multi-label classification problem: $L \geq 2$, $S = 1$, $\beta = 1$, $X \in [0, 1]$.
And let us define two probability measures $\Prob_{-1}, \Prob_{+1}$ which have the marginal distribution $\Prob_{\pm 1, X} \equiv \lambda$, where $\lambda$ is the Lebesgue measure on $[0, 1]$.
Under both $\Prob_{-1}, \Prob_{+1}$ the labels $Y^1, \ldots, Y^L$ are independent,
$Y^l \equiv 0$ for all $l = 3, \ldots, L$, $\Prob_{\rho}(Y^1 = 1 | X) \equiv 3/4$, and $\Prob_{\rho}(Y^2 = 1 | X) = \eta^2_{\rho}(X) = 1/4 - \rho \phi^{-1}_N$ for $\rho \in \{-1, 1\}$ and some strictly increasing sequence $\phi_N$ of $N \in \bbN$.
Assume also that $\phi_N$ is chosen in such a way that $\phi^{-1}_1 \leq 1/8$.
One can see, thanks to Lemma~\ref{lem:oracle_mistakes}, that the Oracle classifiers under $\Prob_{-1}$ and $\Prob_{+1}$ are given by
\begin{align*}
  \oracle^{-1}(x) &= (1, 1, 1, \ldots, 1)^\top\enspace,\\
  \oracle^{+1}(x) &= (1, 0, 0, \ldots, 0)^\top\enspace,
\end{align*}
respectively and the optimal thresholds are $K^{-1}(x) \equiv L$, $K^{+1}(x) \equiv 1$.
Now, let us consider minimax risk\footnote{We are forced to put an absolute value in this discussion, since an "estimator" that always outputs the vector $(1, \ldots, 1)^\top$ achieves zero risk and the minimax risk without the absolute value is always non-positive.} over $\mathcal{P} = \{\Prob_{+1}, \Prob_{-1}\}$ defined as
\begin{align*}
    \inf_{\estimator}\sup_{\Prob \in \mathcal{P}}\Exp_{\Prob^{\otimes N}}\abs{\risk(\estimator) - \risk(\oracle)}
    &=
    \inf_{\estimator}\sup_{\rho \in \{-1, 1\}}\Exp_{\Prob^{\otimes N}_{\rho}}\abs{\mathcal{R}_{\Prob_\rho}(\estimator) - \mathcal{R}_{\Prob_\rho}(\oracle^\rho)}\enspace.
\end{align*}
Then, for the excess risk $\excess_\rho(\estimator) = \abs{\mathcal{R}_{\Prob_\rho}(\estimator) - \mathcal{R}_{\Prob_\rho}(\oracle^\rho)}$ we can write
\begin{align*}
  \excess_\rho(\estimator) = \frac{1}{L}\abs{\frac{3}{4}\int_{0}^1\ind{\estimator^1(x) = 0} dx + \int_{0}^1 \eta^2_{\rho}(x)\left(\ind{\estimator^2(x) = 0} - \ind{(\oracle^\rho)^2(x) = 0}\right)dx}\enspace,
\end{align*}
moreover, using the triangle inequality we can lower bound $\excess_{+1}(\estimator) + \excess_{-1}(\estimator)$ by
\begin{align*}
    \frac{1}{L}\abs{-2\phi^{-1}_N\int_0^1\ind{\estimator^2(x) = 0}dx - \int_{0}^1 \eta^2_{+1}(x)\ind{(\oracle^{+1})^2(x) = 0}  dx+ \int_{0}^1 \eta^2_{-1}(x)\ind{(\oracle^{-1})^2(x) = 0} dx}\enspace.
\end{align*}
Recall that $\enscond{x \in [0, 1]}{(\oracle^{+1})^2(x) = 0} = [0, 1]$ and $\enscond{x \in [0, 1]}{(\oracle^{-1})^2(x) = 0} = \emptyset$, thus
\begin{align*}
  \excess_{+1}(\estimator) + \excess_{-1}(\estimator)
  &\geq
  \frac{1}{L}\abs{-2\phi^{-1}_N\int_0^1\ind{\estimator^2(x) = 0}dx - \frac{1}{4} + \phi^{-1}_N}\\
  &\geq
  \frac{1}{4L} - \frac{1}{L}\abs{-2\phi^{-1}_N\int_0^1\ind{\estimator^2(x) = 0}dx + \phi^{-1}_N}\\
  &\geq
  \frac{1}{4L} - \frac{\phi^{-1}_N}{L}\int_0^1\abs{\ind{\estimator^2(x) = 0} - \ind{\estimator^2(x) = 1}}dx\\
  &=
  \frac{1}{4L} - \frac{\phi^{-1}_N } {L}\geq \frac{1}{8L}\enspace.
\end{align*}
where the last inequality is due to our choice of $\phi_N$.
Now, let us define another probability measure $\Prob_0$ on $[0, 1] \times \{0, 1\}^L$, such that $\Prob_{0, X} \equiv \lambda$ on $[0, 1]$, $Y^1, \ldots, Y^L$ are independent, $Y^l \equiv 0$ for $l = 3, \ldots, L$ and $\Prob_0(Y^1 = 1 | X) \equiv 3/4$, $\Prob_0(Y^2 = 1 | X) \equiv 1/4$.
Importantly, both $\Prob_{-1}$ and $\Prob_{+1}$ are absolutely continuous \wrt to $\Prob_0$.
Let us write $(*) = \inf_{\estimator}\sup_{\rho \in \{-1, 1\}}\Exp_{\Prob^{\otimes N}_{\rho}}\abs{\mathcal{R}_{\Prob_\rho}(\estimator) - \mathcal{R}_{\Prob_\rho}(\oracle^\rho)}$, thus
\begin{align*}
  (*)
  &\geq
  \inf_{\estimator}\frac{\sum_{\rho \in \{-1, 1\}}\Exp_{\Prob^{\otimes N}_{{+1}}}\abs{\mathcal{R}_{\Prob_{\rho}}(\estimator) - \mathcal{R}_{\Prob_{\rho}}(\oracle^{\rho})}}{2}\\
  &=
  \inf_{\estimator}\frac{\Exp_{\Prob^{\otimes N}_{{0}}}\left[\frac{d \Prob^{\otimes N}_{{+1}}}{d\Prob^{\otimes N}_{{0}}}\excess_{+1}(\estimator)\right] + \Exp_{\Prob^{\otimes N}_{{0}}}\left[\frac{d \Prob^{\otimes N}_{{-1}}}{d\Prob^{\otimes N}_{{0}}}\excess_{-1}(\estimator)\right]}{2}\\
  &\geq
  \inf_{\estimator}\frac{\Exp_{\Prob^{\otimes N}_{{0}}}\left[\min\left\{\frac{d \Prob^{\otimes N}_{{+1}}}{d\Prob^{\otimes N}_{{0}}}, \frac{d \Prob^{\otimes N}_{{-1}}}{d\Prob^{\otimes N}_{{0}}}\right\}\left(\excess_{+1}(\estimator) + \excess_{-1}(\estimator) \right)\right]}{2}\\
  &\geq
  \frac{1}{16 L} \left(1 - \text{TV}(\Prob^{\otimes N}_{{+1}}, \Prob^{\otimes N}_{{-1}})\right)
\end{align*}
where $\text{TV}(\cdot, \cdot)$ is the total variation distance between probability measures.
Note, however, that for a sufficiently fast decaying\footnote{One can upper bound the total variation using Pinsker's inequality and observe that the problem reduces to the Kullback-Leibler divergence between two Bernoulli variables.} $\phi_N$ we can guarantee that $\text{TV}(\Prob^{\otimes N}_{+1}, \Prob^{\otimes N}_{-1}) \leq 1/2$, which implies that
\begin{align*}
  \inf_{\estimator}\sup_{\Prob \in \mathcal{P}}\Exp_{\Prob^{\otimes N}}\abs{\risk(\estimator) - \risk(\oracle)} \geq \frac{1}{16 L}\enspace.
\end{align*}
First of all, observe that the distributions constructed above satisfy Assumptions~\ref{ass:separability},~\ref{ass:sparsity_general} and the corresponding regression functions are constant.
In particular, these regression functions are infinitely many times differentiable, the marginal distribution admits density \wrt Lebesgue measure supported on $[0, 1]$, and an estimator achieving Assumption~\ref{ass:exponential_heta} exists.
However, since the minimax risk is of constant order, it suggests that an extra assumption is necessary for consistency of any estimator.
This phenomena occurs due to the behavior of the regression function around the parameter $\beta$.
The discussion above highlights the fundamental difference between the two frameworks considered in this work and motivates the extra Assumption~\ref{ass:sparsity_struct}, which might seem to be unnatural at the first sight.

Using the assumptions introduced for this model we can state the following result.
\begin{theorem}
    \label{thm:main_mistakes}
    Assume that the estimator $\heta$ satisfies Assumption~\ref{ass:exponential_heta}.
    Therefore, under Assumptions~\ref{ass:separability}--\ref{ass:sparsity_struct}, the plug-in rule in Eq.~\eqref{eq:plug-in_rule} satisfies
    \[
        \Expdata\big[\risk(\estimator)\big] - \risk(\oracle) \leq \tilde{C}(\beta^{\alpha_2} + S)\frac{(L - \beta)}{L}N^{-\gamma(\alpha_2 \wedge \alpha_1)/2}\enspace,
    \]
    for some universal constant $\tilde{C}$.
\end{theorem}
The proof of the previous theorem relies on the following Lemma:
\begin{lemma}[Partial order]\label{prop:partial_ordering_of_estimator}
    On the event $\{2\norm{\eta(X) - \heta(X)}_\infty < \eta^{\sigma_k}(X) - \eta^{\sigma_{k + 1}}(X)\}$, we have for all $l \in [L]$ and all $m \in [L]$ such that $l \leq k < m$:
    \[
        l' \leq k < m'\enspace,
    \]
    where $l'$ and $m'$ are such that $\hsigma_{l'} = \sigma_l$ and $\hsigma_{m'} = \sigma_{m}$.
\end{lemma}

From the previous result we can conclude that the condition of the lemma yields $\{\hsigma_1, \ldots, \hsigma_k\} = \{\sigma_1, \ldots, \sigma_k\}$.
To see this it is sufficient to apply Lemma~\ref{prop:partial_ordering_of_estimator} to each $l = 1, \ldots, k$ and $m = k+1$ and use the fact that $l'$ defined in Lemma~\ref{prop:partial_ordering_of_estimator} is unique and is different for all $l$.
Similarly we can show that $\{\hsigma_{k+1}, \ldots, \hsigma_{L}\} = \{\sigma_{k+1}, \ldots, \sigma_{L}\}$.
Hence, the previous lemma gives an intuitive result: if the estimation $\heta$ is accurate enough, then it partially preserves the ordering of $\eta$.
We again point out the weak dependence of the obtained bound on the total amount of labels $L$, since $\beta$ and $S$ are expected to be small compared to $L$.

\section{Discussions and conclusions}
\label{sec:discussion_extensions}

The bound in Theorem~\ref{main:sparse} is similar to the bound obtained by \citet{Audibert_Tsybakov07} in the binary classification settings and is known to be minimax optimal in this case.
It is important to notice that this bound is independent from the total amount of labels $L$ and only depends on the parameter $K$.
However, we expect that this dependency can be improved, and rates proportional to $K/L$ can be achieved.
This intuition is explained by the first step of the proof of Theorem~\ref{main:sparse}, where a rather loose inequality is used.

Theorem~\ref{thm:main_mistakes} is proven under three different assumptions, which are reflecting the structure of the regression vector $\eta$.
The obtained bound does not have a classical $\gamma(\alpha+1)/2$ rate but $\gamma\alpha/2$ is obtained instead.
A simple explanation for this phenomena can be provided: in case the constant $\alpha_2$ from Assumption~\ref{ass:sparsity_struct} is equal to zero, the upper-bound becomes trivial, it is not surprising in view of the discussion provided after Assumption~\ref{ass:sparsity_struct}.
Indeed, the distributions satisfying Assumption~\ref{ass:sparsity_struct} with $\alpha_2 = 0$ are the same as the distributions satisfying \emph{only} Assumptions~\ref{ass:separability},~\ref{ass:sparsity_general} which is not sufficient for upper-bounding the minimax risk.
Bound of a similar type can be found in~\citep{Denis_Hebiri15} in the case of binary classification with reject option, where the authors are proposing to control the probability of rejection.
We expect that the control over a random quantity, that is the false positive discovery in our case, or the probability of reject in case of~\citep{Denis_Hebiri15}, might lead to such rates.
We plan to further investigate the behavior obtained in this case and provide minimax lower bounds to show their optimality.

The proposed framework is flexible and could be further analyzed.
In particular, it is interesting to find other strategies, \ie other sets $\classgeneric$ that can be of practical interest.
In general, we suggest to incorporate any quantity of interest in the set $\classgeneric$ and consider the plug-in approach if the oracle is available explicitly.
Notice, that the $\classtop{K}$-oracle does not allow to have an optimal control over false positive discoveries, whereas the $\classbeta{\beta}$-oracle does not allow to control sparsity.
For instance, one might be interested in both sparsity and false positive discoveries simultaneously, to this end a natural extension is the following set $\classmixer{\beta}{K}$:
\begin{align*}
    \classmixer{\beta}{K} = \classtop{K} \cap \classbeta{\beta}\enspace.
\end{align*}
Each classifier in this set has a controlled number of false positive errors as well as bounded sparsity.
Moreover, Remark~\ref{rem:on_fp_sp_inclusion} suggests to choose the value $\beta$ as $\rho K$, where $\rho \in (0, 1)$.
This choice of parameters bounds the output sparsity from below on the level $\rho K$ and from above on the level $K$, moreover the false positives discoveries are upper-bounded by $\rho K$.
It is not hard to show, that an $\classmixer{\beta}{K}$-oracle over such set can be obtained in a similar fashion, that is:
\begin{align*}
      \oracle^{\sigma_1}(x) &= \ldots = \oracle^{\sigma_{K_\star}}(x) = 1\enspace,\\
      \oracle^{\sigma_{K_\star+1}}(x) &= \ldots = \oracle^{\sigma_L}(x) = 0\enspace,
\end{align*}
where $K_{\star} = K_\star(x)$ is defined as
\begin{align*}
    K_\star(x) = \max\big\{m \in [L] : \sum\limits_{l = 1}^m(1 - \eta^{\sigma_l}(x)) \leq \beta\big\}\wedge K\enspace.
\end{align*}
The explicit expression of the $\classmixer{\beta}{K}$-oracle allows to use plug-in approach as before, and we plan to investigate this strategy in future works.


We presented a generic framework for multi-label classification, which consists in the minimization of the amount of false negative discoveries over a set of classifiers.
We provided non-asymptotic excess risk bounds for two particular instances of the proposed framework.
As a future direction, we plan to provide a lower bound for the considered examples and further investigate on other possible strategies within the proposed framework.
An important direction is to study the optimality of the obtained rates with respect to the total amount of labels $L$.
\vskip 0.2in
\bibliographystyle{elsarticle-harv}
\bibliography{references_all}
\newpage

\appendix
\section*{Technical lemmas}
\label{app:theorem}
The following lemma is used throughout this work.
It ensures that if the estimate $\heta$ satisfies the exponential bound in Assumption~\ref{ass:exponential_heta}, hence the same bound holds if we replace $l \in [L]$ by $\sigma_j$ for every $j \in [L]$:
\begin{lemma}
    \label{eq:exponential_heta_order}
    Assume that $\heta$ satisfies the conditions in Assumption~\ref{ass:exponential_heta}, hence for all $j \in [L]$ we have
    \begin{align*}
      \Probdata\{|\eta^{\sigma_j}(x) - \heta^{\sigma_j}(x)| \geq \delta\} \leq C_1\exp(-C_2a_N \delta^2)\,\,\, \text{for almost every $x \in \bbR^D$ w.r.t. }\Prob_X\enspace.
    \end{align*}
\end{lemma}
\begin{proof}
    A standard disjunction yields:
    \begin{align*}
        \Probdata\{|\eta^{\sigma_j}(x) - \heta^{\sigma_j}(x)| \geq \delta\} &= \sum\limits_{l = 1}^L\Probdata\{|\eta^{\sigma_j(x)}(x) - \heta^{\sigma_j(x)}(x)| \geq \delta\} \ind{\sigma_j(x) = l}\\
        &=\sum\limits_{l = 1}^L\Probdata\{|\eta^{l}(x) - \heta^{l}(x)| \geq \delta\} \ind{\sigma_j(x) = l}\\
        &\leq \sum\limits_{l = 1}^L\ C_1\exp(-C_2a_N \delta^2)\ind{\sigma_j(x) = l} = C_1\exp(-C_2a_N \delta^2)\enspace,
    \end{align*}
    and the inequality in Lemma~\ref{eq:exponential_heta_order} holds for almost every $x \in \bbR^D$ with respect to $\Prob_X$.
\end{proof}
Similarly, we can obtain the following bound on the infinity norm of the regression function:
\begin{lemma}
    Assume that $\heta$ satisfies the conditions in Assumption~\ref{ass:exponential_heta}, hence we have
    \begin{align*}
      \Probdata\{\max_{l \in [L]}\{|\eta^{l}(x) - \heta^{l}(x)|\} \geq \delta\} \leq C_1\exp(-C_2a_N \delta^2)\,\,\, \text{for almost every $x \in \bbR^D$ w.r.t. }\Prob_X\enspace.
    \end{align*}
\end{lemma}

\section*{Proof of Theorem~\ref{main:sparse}}

\begin{proof}
    We start with the following decomposition of the excess risk:
    \begin{align*}
        \excess(\estimator) &= \frac{1}{L}\Expdata\Exp_{\Prob_X}\Big[\sum\limits_{l = 1}^L\eta^{\sigma_l}(X)\big(\ind{\estimator^{\sigma_l}(X) = 0} - \ind{\oracle^{\sigma_l}(X) = 0}\big)\Big]\\
        &= \frac{1}{L}\Expdata\Exp_{\Prob_X}\Big[\sum\limits_{l = 1}^L\eta^{\sigma_l}(X)\ind{\estimator^{\sigma_l}(X) = 0, \oracle^{\sigma_l}(X) = 1} - \sum\limits_{l = 1}^L\eta^{\sigma_l}(X)\ind{\estimator^{\sigma_l}(X) = 1, \oracle^{\sigma_l}(X) = 0}\Big]\\
        &=\frac{1}{L}\Expdata\Exp_{\Prob_X}\Big[\sum\limits_{l = 1}^K\eta^{\sigma_l}(X)\ind{\estimator^{\sigma_l}(X) = 0} - \sum\limits_{l = K + 1}^L\eta^{\sigma_l}(X)\ind{\estimator^{\sigma_l}(X) = 1}\Big]\enspace,
    \end{align*}
    where in the last equality we have used the explicit expression for the oracle from Lemma~\ref{lem:oracle_sparse}.
    Now notice that since $\estimator$ is exactly $K$-sparse, hence if in the first sum there are $m \in \{1, \ldots, K\}$ non-zero terms, hence there are exactly $m$ non-zero terms in the second sum.
    Since all the non-zero terms in the first sum are greater than all the non-zero terms in the second sum, we can bound the excess risk by all possible pair-wise differences:
    \begin{align*}
        \excess(\estimator) &\leq \Expdata\Exp_{\Prob_X}\frac{1}{L}\sum\limits_{l=1}^K\sum\limits_{j = K+1}^{L}(\eta^{\sigma_l}(X) - \eta^{\sigma_j}(X))\ind{\estimator^{\sigma_l}(X) = 0, \estimator^{\sigma_j}(X) = 1}\enspace.
    \end{align*}
    On the one hand, according to the plug-in rule definition, on the event $\{\estimator^{\sigma_l}(X) = 0, \estimator^{\sigma_j}(X) = 1\}$ we have $\heta^{\sigma_j}(X) \geq \heta^{\sigma_l}(X)$.
    On the other hand, due to the definition of $\sigma$ we have $\eta^{\sigma_l}(X) \geq \eta^{\sigma_j}(X)$ for all $j > l$.
    Therefore, on the event $\{\estimator^{\sigma_l}(X) = 0, \estimator^{\sigma_j}(X) = 1\}$ we have $\eta^{\sigma_l}(X) - \eta^{\sigma_j}(X) \leq |\hat\Delta_l| + |\hat\Delta_j|$, where $\hat\Delta_k=\heta^{\sigma_k}(X) - \eta^{\sigma_k}(X)$ for any $k \in [L]$.
    We can write:
    \begin{align*}
        \excess(\estimator) &\leq \frac{1}{L}\Expdata\Exp_{\Prob_X}\sum\limits_{l=1}^K\sum\limits_{j = K+1}^{L}(\eta^{\sigma_l}(X) - \eta^{\sigma_j}(X))\ind{\eta^{\sigma_l}(X) - \eta^{\sigma_j}(X) \leq |\hat\Delta_k| + |\hat\Delta_j|}\enspace.
    \end{align*}
    Denote by $T_{l,j}(X)$ for all $l \in \{1, \ldots, K\}$ and $j \in \{K+1, \ldots, L\}$ the $(l,j)$-term in the above sum, that is:
    \begin{align*}
        \excess(\estimator) &\leq \frac{1}{L}\sum\limits_{l=1}^K\sum\limits_{j = K+1}^{L}\Expdata\Exp_{\Prob_X}T_{l, j}(X)\enspace.
    \end{align*}
    Now we restrict our attention on an arbitrary $T_{l, j}(X)$, we can write
    \begin{align*}
      \Expdata\Exp_{\Prob_X}T_{l, j}(X) &= \Expdata\Exp_{\Prob_X}(\eta^{\sigma_l}(X) - \eta^{\sigma_j}(X))\ind{\eta^{\sigma_l}(X) - \eta^{\sigma_j}(X) \leq |\hat\Delta_l| + |\hat\Delta_j|}\\
      &= \sum\limits_{p \geq 0}\Expdata\Exp_{\Prob_X}(\eta^{\sigma_l}(X) - \eta^{\sigma_j}(X))\ind{\eta^{\sigma_l}(X) - \eta^{\sigma_j}(X) \leq |\hat\Delta_l| + |\hat\Delta_j|}\ind{X \in A_p}\enspace,
    \end{align*}
    where $A_p$ are defined similar to~\citep{Audibert_Tsybakov07}, that is
    \begin{align*}
        A_0 &= \{x \in \bbR^D: 0 < \eta^{\sigma_l}(x) - \eta^{\sigma_j}(x) \leq \delta\}\enspace,\\
        A_p &= \{x \in \bbR^D: 2^{p - 1} < \eta^{\sigma_l}(x) - \eta^{\sigma_j}(X) \leq 2^p\delta\}\text{ for all $p > 0$}\enspace.
    \end{align*}
    We continue as:
    \begin{align*}
      \Expdata&\Exp_{\Prob_X}T_{l, j}(X) = \Expdata\Exp_{\Prob_X}(\eta^{\sigma_l}(X) - \eta^{\sigma_j}(X))\ind{\eta^{\sigma_l}(X) - \eta^{\sigma_j}(X) \leq |\hat\Delta_l| + |\hat\Delta_j|}\ind{X \in A_0}\\
      &+ \sum\limits_{p \geq 1}\Expdata\Exp_{\Prob_X}(\eta^{\sigma_l}(X) - \eta^{\sigma_j}(X))\ind{\eta^{\sigma_l}(X) - \eta^{\sigma_j}(X) \leq |\hat\Delta_l| + |\hat\Delta_j|}\ind{X \in A_p}\\
      \leq &\Expdata\Exp_{\Prob_X}(\eta^{\sigma_l}(X) - \eta^{\sigma_j}(X))\ind{0 < \eta^{\sigma_l}(X) - \eta^{\sigma_j}(X) \leq \delta\}}\\
      &+\sum\limits_{p \geq 1}\Expdata\Exp_{\Prob_X}(\eta^{\sigma_l}(X) - \eta^{\sigma_j}(X))\ind{2^{p - 1}\delta \leq |\hat\Delta_l| + |\hat\Delta_j|}\ind{0 < \eta^{\sigma_l}(X) - \eta^{\sigma_j}(X) \leq 2^p\delta}\\
      \leq &\delta\Exp_{\Prob_X}\ind{0 < \eta^{\sigma_l}(X) - \eta^{\sigma_j}(X) \leq \delta} \\
      &+ \sum\limits_{p \geq 1}2^p\delta\Exp_{\Prob_X}\Probdata\{2^{p - 1}\delta \leq |\hat\Delta_l| + |\hat\Delta_j|\}\ind{0 < \eta^{\sigma_l}(X) - \eta^{\sigma_j}(X) \leq 2^p\delta} \\
      \leq &\delta\Exp_{\Prob_X}\ind{0 < \eta^{\sigma_{K}}(X) - \eta^{\sigma_{K+1}}(X) \leq \delta} \\
      &+ \sum\limits_{p \geq 1}2^p\delta\Exp_{\Prob_X}\Probdata\{2^{p - 1}\delta \leq |\hat\Delta_l| + |\hat\Delta_j|\}\ind{0 < \eta^{\sigma_{K}}(X) - \eta^{\sigma_{K+1}}(X) \leq 2^p\delta} \\
      \leq &C\delta^{1+\alpha} + 2\sum\limits_{p \geq 1}2^p\delta C_1\exp(-C_2a_N 2^{2p - 2}\delta^2)\Exp_{\Prob_X}\ind{0 < \eta^{\sigma_{K}}(X) - \eta^{\sigma_{K+1}}(X) \leq 2^p\delta}\\
      \leq &C\delta^{1+\alpha} + 2CC_1 \delta^{\alpha + 1} \sum\limits_{p \geq 1}2^{p(\alpha+1)}\exp(-C_2a_N 2^{2p - 2}\delta^2)\enspace,
    \end{align*}
    setting $\delta = (a_N)^{\tfrac{1}{2}}$ we obtain:
    \begin{align*}
        \Expdata\Exp_{\Prob_X}T_{l, j}(X) \leq \tilde{C} (a_N)^{\tfrac{1 + \alpha}{2}}.
    \end{align*}
    We conclude by substituting the obtained bound into the excess risk bound.
\end{proof}
\section*{Proof of Lemma~\ref{lem:oracle_mistakes}}
\begin{lemma}
    \label{claim:pointwise_top_oracle}
    Assume that $f \in \mathcal{F}_\beta$.
    Let $\pos(f(x)) = \{l \in [L] : f^l(x) = 1\}$ and denote by $m(x) = |\pos(f(x))|$ the cardinality of $\pos(f(x))$.
    For all $x\in \bbR^D$ define $f_m$ as
    \begin{align*}
        f^{\sigma_1}_m(x) &= \ldots = f^{\sigma_m}_m(x) = 1\enspace,\\
        f^{\sigma_{m+1}}_m(x) &= \ldots = f^{\sigma_L}_m(x) = 0\enspace,
    \end{align*}
    Hence, $f_m \in \classbeta\beta$ and $\risk(f_m) \leq \risk(f)$.
\end{lemma}
\begin{proof}
    First, we show that $f_m \in \classbeta\beta$, since $f \in \mathcal{F}_\beta$ it holds that
    \[
        \sum\limits_{l \in \pos(f(X))}(1 - \eta^l(X)) \leq \beta,\,\asx\enspace,
    \]
    due to the definition of $\sigma = \sigma(x)$ we have
    \[
       \sum\limits_{l \in \pos(f_m(x))}(1 - \eta^l(x)) \leq \sum\limits_{l \in \pos(f(x))}(1 - \eta^l(x)),\,\allx\enspace,
    \]
    indeed, $\sum_{l \in \pos(f_m(x))}^L(1 - \eta^l(x))$ consists of a sum of $m$ smallest values of $(1 - \eta^l(x))_{l=1}^L$, which concludes the first part of the statement.
    The second part is proven similarly: $\allx$ it obviously holds thanks to the definition of $\sigma$ that
    \begin{align*}
        \Exp\Big[\sum\limits_{l = 1}^L\ind{f^l_m(x) = 0, Y^l = 1}| X = x\Big] = \sum\limits_{l \in [L] \setminus \pos(f_m(x))}\eta^l(x) &\leq \sum\limits_{l \in [L] \setminus \pos(f(x))}\eta^l(x)\\ &= \Exp\Big[\sum\limits_{l = 1}^L\ind{f^l(x) = 0, Y^l = 1}| X = x\Big]\enspace,
    \end{align*}
    which concludes the proof.
\end{proof}

\begin{proof}[Lemma~\ref{lem:oracle_mistakes}]
    Let $f \in \classbeta\beta$ be an oracle. Due to Lemma~\ref{claim:pointwise_top_oracle}, we can get $f_m \in \classbeta\beta$ such that $\risk(f_m)\leq \risk(f)$, so $f_m$ is also an oracle.
    On the event $\{ x \in \bbR^D : \sum_{l = 1}^L\ind{f^l_m(x) = 1}(1 - \eta^l(x)) \leq \beta\}$ (whose measure is one), it holds that $\pos(f_m(x)) \subset \pos(\oracle(x))$ by the construction of $\oracle(x)$ and in particular $K(x)$ therefore on this set
    \begin{align*}
        \Exp\Big[\sum\limits_{l = 1}^L\ind{f^l_m(x) = 0, Y^l = 1}| X = x\Big] \leq \Exp\Big[\sum\limits_{l = 1}^L\ind{f^l_\beta(x) = 0, Y^l = 1}| X = x\Big]\enspace.
    \end{align*}
    Since, the previous inequality holds almost surely $\Prob_X$, we conclude.

\end{proof}

\section*{Proof of Theorem~\ref{thm:main_mistakes}}
\label{app:proof_main_mistakes}

\begin{proof}[Lemma~\ref{lem:embeding_of_plug-in_beta_class}]
    Let us fix $f \in \hclassbeta{\beta}$.
    Hence, by the definition of $\hclassbeta{\beta}$ we have
    \[
        \sum\limits_{l = 1}^L\ind{f^l(X) = 1}(1 - \heta^l(X))\leq \beta,\,\asx \enspace.
    \]
    We introduce the following notation
    \begin{align*}
      B(z) = \sum\limits_{l = 1}^L\ind{f^l(X) = 1}(1 - z^l)\enspace.
    \end{align*}
    Therefore $B(\heta(X))\leq \beta$, $\asx$. The following sequence of inequalities holds
    \begin{align*}
      \abs{B(\heta(X)) - B(\eta(X))} &=
      \abs{\sum\limits_{l = 1}^L\ind{f^l(X) = 1}(\heta^l(X) - \eta^l(X))}\\
      &\leq \sum\limits_{l = 1}^L\abs{(\heta^l(X) - \eta^l(X))},\,\asx\enspace,
    \end{align*}
    and this concludes the proof.
\end{proof}
\begin{proof}[Lemma~\ref{prop:partial_ordering_of_estimator}]
    Let $l \in [L]$ and $m \in [L]$ be such that $l \leq k < m$, hence by of $\sigma$ we have
    \[
        \eta^{\sigma_l}(X) \geq \eta^{\sigma_k}(X) \geq \eta^{\sigma_{k + 1}}(X) \geq \eta^{\sigma_m}(X)\enspace,
    \]
    therefore
    \[
        \eta^{\sigma_l}(X) - \eta^{\sigma_m}(X) \geq \eta^{\sigma_k}(X) - \eta^{\sigma_{k+1}}(X)\enspace.
    \]
    We can write
    \[
       \eta^{\sigma_l}(X) - \eta^{\sigma_m}(X) - \heta^{\hsigma_{l'}}(X) + \heta^{\hsigma_{l'}}(X) - \heta^{\hsigma_{m'}}(X) + \heta^{\hsigma_{m'}}(x) = \eta^{\sigma_l}(X) - \eta^{\sigma_m}(x) \geq \eta^{\sigma_k}(X) - \eta^{\sigma_{k+1}}(X)\enspace,
    \]
    which implies
    \[
        \heta^{\hsigma_{l'}}(X) - \heta^{\hsigma_{m'}}(X) + 2\norm{\heta(X) - \eta(X)}_{\infty} \geq \eta^{\sigma_k}(X) - \eta^{\sigma_{k+1}}(X)\enspace,
    \]
    and therefore on the event $\{2\norm{\eta(X) - \heta(X)}_\infty < \eta^{\sigma_k}(X) - \eta^{\sigma_{k + 1}(X)}\}$ we have
    \[
        \heta^{\hsigma_{l'}}(X) \geq \heta^{\hsigma_{m'}}(X)\enspace,
    \]
    meaning that $l' < m'$.
    To conclude that $m' > k$ it is sufficient to notice that the inequality $l' < m'$ holds for at least $k$ different values of $l'$.
    Similarly we conclude that $l' \leq k$.
\end{proof}

\begin{proof}[Theorem~\ref{thm:main_mistakes}]
Here we denote by $\Exp$ the expectation $\Expdata\Exp_X$ for the sake of simplicity.
    \begin{align*}
        \Exp\risk(\estimator) - &\risk(\oracle) = \frac{1}{L}\Exp\big[\sum\limits_{l=1}^L\eta^{\sigma_l}(X)\ind{\estimator^{\sigma_l}(X) =0, \oracle^{\sigma_l}(X)=1} - \sum\limits_{l=1}^L\eta^{\sigma_l}(X)\ind{\estimator^{\sigma_l}(X) =1, \oracle^{\sigma_l}(X)=0}\big] \\
        \leq &\frac{1}{L}\Exp\big[\sum\limits_{l=1}^L\eta^{\sigma_l}(X)\ind{\estimator^{\sigma_l}(X) =0, \oracle^{\sigma_l}(X)=1}\sum\limits_{k = \beta}^L\ind{K(X) = k}\big]\\
        = &\underbrace{\frac{1}{L}\Exp\Big[\sum\limits_{k = \beta}^L\big[\sum\limits_{l=1}^k\eta^{\sigma_l}(X)\ind{\estimator^{\sigma_l}(X) = 0}\ind{K(X) = k}\big]\ind{2\norm{\eta(X) - \heta(X)}_\infty \geq \eta^{\sigma_k}(X) - \eta^{\sigma_{k + 1}(X)}}\Big]}_{U_1} \\
        + &\underbrace{\frac{1}{L}\Exp\Big[\sum\limits_{k = \beta}^L\big[\sum\limits_{l=1}^k\eta^{\sigma_l}(X)\ind{\estimator^{\sigma_l}(X) = 0}\ind{K(X) = k}\big]\ind{2\norm{\eta(X) - \heta(X)}_\infty < \eta^{\sigma_k}(X) - \eta^{\sigma_{k + 1}(X)}}\Big]}_{U_2} \\
        = & U_1 + U_2\enspace.
    \end{align*}
    For $U_1$, due to Assumption~\ref{ass:sparsity_general} we can write
    \begin{align*}
        U_1 \leq &\frac{1}{L}\Exp\Big[\sum\limits_{k = \beta}^L\big[\sum\limits_{l=1}^k\eta^{\sigma_l}(X)\ind{K(X) = k}\ind{2\norm{\eta(X) - \heta(X)}_\infty \geq \eta^{\sigma_k}(X) - \eta^{\sigma_{k + 1}(X)}}\big]\Big]\\
        = &\frac{1}{L}\Exp\Big[\sum\limits_{k = \beta}^L\ind{K(X) = k}\ind{2\norm{\eta(X) - \heta(X)}_\infty \geq \eta^{\sigma_k}(X) - \eta^{\sigma_{k + 1}(X)}}\sum\limits_{l=1}^k\eta^{\sigma_l}(X)\Big]\\
        \leq &\frac{S}{L}\sum\limits_{k = \beta}^L\underbrace{\Exp\Big[\ind{K(X) = k}\ind{2\norm{\eta(X) - \heta(X)}_\infty \geq \eta^{\sigma_k}(X) - \eta^{\sigma_{k + 1}(X)}}\Big]}_{U_1^k}
        = \frac{S}{L}\sum\limits_{k = \beta}^LU_1^k\enspace.
    \end{align*}
    define the following sets, similar to the analysis of~\citet{Audibert_Tsybakov07} for binary classification, for all $L \leq k \leq \beta$
    \begin{align*}
        A_0^k &= \{X \in \bbR^D :  \eta^{\sigma_k}(X) - \eta^{\sigma_{k + 1}(X)} \leq \delta\}\enspace,\\
        A_j^k &= \{X \in \bbR^D :  2^{j - 1}\delta < \eta^{\sigma_k}(X) - \eta^{\sigma_{k + 1}}(X) \leq 2^j\delta\}\enspace.
    \end{align*}
    Therefore, using Assumption~\ref{ass:separability} for each $k$ such that $L \leq k \leq \beta$ we have
    \begin{align*}
        U_1^k = &\Exp\sum\limits_{j \geq 0}\ind{K(X) = k}\ind{2\norm{\eta(X) - \heta(X)}_\infty \geq \eta^{\sigma_k}(X) - \eta^{\sigma_{k + 1}(X)}}\ind{X \in A_j^k}\\
        = & \Exp\ind{K(X) = k}\ind{2\norm{\eta(X) - \heta(X)}_\infty \geq \eta^{\sigma_k}(X) - \eta^{\sigma_{k + 1}(X)}}\ind{X \in A_0^k} \\
        &+ \Exp\sum\limits_{j \geq 1}\ind{K(X) = k}\ind{2\norm{\eta(X) - \heta(X)}_\infty \geq \eta^{\sigma_k}(X) - \eta^{\sigma_{k + 1}(X)}}\ind{X \in A_j^k} \\
        \leq &\Exp\ind{K(X) = k}\ind{X \in A_0^k} \\
        &+ \Exp\sum\limits_{j \geq 1}\ind{K(X) = k}\ind{2\norm{\eta(X) - \heta(X)}_\infty \geq 2^{j - 1}\delta}\ind{X \in A_j^k} \\
        \leq &\Prob\{0 < \eta^{\sigma_k}(X) - \eta^{\sigma_{k + 1}(X)} \leq \delta, K(X) = k\} \\
        &+\Exp\sum\limits_{j \geq 1}\ind{\eta^{\sigma_k}(X) - \eta^{\sigma_{k + 1}}(X) \leq 2^j\delta, K(X) = k}\Probdata\{2\norm{\eta(X) - \heta(X)}_\infty \geq 2^{j - 1}\delta\} \\
        \leq &\Prob\{0 < \eta^{\sigma_k}(X) - \eta^{\sigma_{k + 1}(X)} \leq \delta, K(X) = k\} \\
        &+\sum\limits_{j \geq 1}\Prob\{\eta^{\sigma_k}(X) - \eta^{\sigma_{k + 1}}(X) \leq 2^j\delta, K(X) = k\}C_2\exp(-C_3N^{\gamma}2^{2j - 2}\delta^2) \\
        \leq &C_1\delta^{\alpha_1} + \sum\limits_{j \geq 1}C_1C_2\delta^\alpha 2^{\alpha_1 j}\exp(-C_3N^{\gamma}2^{2j - 2}\delta^2)\enspace,
    \end{align*}
    let $N^{-\gamma/2}$, hence
    \begin{align*}
        U_1^k \leq C_1N^{-\gamma\alpha_1/2} + C_1C_2N^{-\gamma\alpha_1/2}\sum\limits_{j \geq 1}2^{\alpha j}\exp(-C_32^{2j - 2}) \leq  \tilde{C}_1 N^{-\gamma\alpha_1/2}\enspace.
    \end{align*}
    Therefore,
    \[
        U_1 \leq \tilde{C}\frac{S}{L}(L - \beta)L a_n^{-\alpha/2} = \tilde{C} S \frac{(L - \beta)}{L}N^{-\gamma\alpha_1/2}\enspace.
    \]
    For $U_2$ we can write
    \begin{align*}
        U_2 &= \frac{1}{L}\Exp\Big[\sum\limits_{k = \beta}^L\big[\sum\limits_{l=1}^k\eta^{\sigma_l}(X)\ind{\estimator^{\sigma_l}(X) = 0}\big]\ind{K(X) = k}\ind{2\norm{\eta(X) - \heta(X)}_\infty < \eta^{\sigma_k}(X) - \eta^{\sigma_{k + 1}(X)}}\Big]\\
        &=\frac{1}{L}\sum\limits_{k = \beta}^LU_2^k\enspace,
    \end{align*}
    where $U_2^k$ is given as
    \[
        U_2^k = \Exp\Big[\sum\limits_{l=1}^k\eta^{\sigma_l}(X)\ind{\estimator^{\sigma_l}(X) = 0}\ind{K(X) = k}\ind{2\norm{\eta(X) - \heta(X)}_\infty < \eta^{\sigma_k}(X) - \eta^{\sigma_{k + 1}(X)}}\Big]\enspace.
    \]
    For each $U_2^k$ we can write
    \begin{align*}
        U_2^k \leq \Exp\Big[\sum\limits_{l=1}^k\ind{\estimator^{\sigma_l}(X) = 0}\ind{K(X) = k}\ind{2\norm{\eta(X) - \heta(X)}_\infty < \eta^{\sigma_k}(X) - \eta^{\sigma_{k + 1}(X)}}\Big]\enspace.
    \end{align*}
    If $\estimator^{\sigma_l}(X) = 0$ for some $l \leq k$, hence by the definition of the plug-in rule we have
    \[
        \sum\limits_{j = 1}^{l'}(1 - \heta^{\hsigma_j}(X)) > \beta\enspace,
    \]
    where $l'$ is such that $\hsigma_{l'} = \sigma_l$.
    Additionally, on the event $\{2\norm{\eta(X) - \heta(X)}_\infty < \eta^{\sigma_k}(X) - \eta^{\sigma_{k + 1}(X)}\}$ according to Lemma~\ref{prop:partial_ordering_of_estimator} we have $\{\hsigma_1, \ldots, \hsigma_{l'}\} \subset \{\sigma_1, \ldots, \sigma_{k}\}$ and hence, on the event $\{K(X) = k\}$ we can write
    \[
        \underbrace{\sum\limits_{j = 1}^{l'}(1 - \eta^{\hsigma_j}(X))}_{\text{Sum of $l'$ elements}} \leq \underbrace{\sum\limits_{j = k - l' + 1}^{k}(1 - \eta^{\sigma_j}(X))}_{\text{Sum of the largest $l'$ elements}} \leq \beta\enspace.
    \]
    Therefore on the intersection of the three events $$\left\{\estimator^{\sigma_l}(X) = 0\right\}\cap\left\{2\norm{\eta(X) - \heta(X)}_\infty < \eta^{\sigma_k}(X) - \eta^{\sigma_{k + 1}(X)}\right\}\cap\left\{K(X) = k\right\}$$ we have
    \[
        l'\norm{\eta(X) - \heta(X)}_{\infty}\geq \abs{\sum\limits_{j = 1}^{l'}\heta^{\hsigma_j}(X) - \eta^{\hsigma_j}(X)} \geq |\sum\limits_{j = k - l' + 1}^{k}(1 - \eta^{\sigma_j}(X)) - \beta|\enspace.
    \]
    Hence,
    \begin{small}

    \begin{align*}
        U_2^k &\leq \Exp\Big[\sum\limits_{l=1}^k\ind{l\norm{\eta(X) - \heta(X)}_{\infty} \geq |\sum\limits_{j = k - l + 1}^{k}(1 - \eta^{\sigma_j}(X)) - \beta|}\ind{K(X) = k}\ind{2\norm{\eta(X) - \heta(X)}_\infty < \eta^{\sigma_k}(X) - \eta^{\sigma_{k + 1}}(X)}\Big]\\
        &\leq \Exp\Big[\sum\limits_{l=1}^k\ind{l\norm{\eta(X) - \heta(X)}_{\infty} \geq |\sum\limits_{j = k - l + 1}^{k}(1 - \eta^{\sigma_j}(X)) - \beta|}\ind{K(X) = k}\enspace.
    \end{align*}
    \end{small}
    where the first inequality is obtained by reordering thanks to Lemma~\ref{prop:partial_ordering_of_estimator}.
    With Assumption~\ref{ass:sparsity_struct}, we can show the following bound, using the same technique as for $U_1^k$:
    \[
        U_2^k \leq \tilde{C}\beta^{\alpha_2}N^{-\gamma\alpha_2/2}\sum\limits_{l =1}^L h(|k - l|)\enspace,
    \]
    therefore
    \[
        U_2 \leq \tilde{C}\beta^{\alpha_2}N^{-\gamma\alpha_{2}/2}\frac{1}{L}\sum\limits_{k = \beta}^L\sum\limits_{l =1}^L h(|k - l|) \leq \tilde{C}\beta^{\alpha_2}\frac{(L - \beta)}{L}N^{-\gamma\alpha_2/2}\enspace.
    \]
    Therefore, we have
    \begin{align*}
        \Exp\risk(\estimator) &\leq \risk(\oracle) + \tilde{C}\frac{(L - \beta)}{L}(\beta^{\alpha_2}N^{-\gamma\alpha_2/2} + SN^{-\gamma\alpha_1/2})\\ &\leq \risk(\oracle) + 2\tilde{C}(\beta^{\alpha_2} + S)\frac{(L - \beta)}{L}N^{-\gamma(\alpha_2 \wedge \alpha_1)/2}\enspace,
    \end{align*}
    and the conclusion holds.
\end{proof}

\end{document}